\documentclass{article}
\usepackage{amsmath,amsthm,amssymb}
\usepackage{graphicx}
\usepackage{hyperref}

\newtheorem{theorem}{Theorem}
\newtheorem{proposition}{Proposition}
\newtheorem*{corollary}{Corollary}
\newtheorem*{lemma}{Lemma}

\theoremstyle{definition}
\newtheorem*{definition}{Definition}

\begin{document}

\title{Angle Bisectors of a Triangle in Lorentzian Plane}
\author{Joseph Cho}
\date{August 25, 2013}
\maketitle

\begin{abstract}
In Lorentzian geometry, limited definition of angles restricts the use of angle bisectors in study of triangles. This paper redefines angle bisectors so that they can be used to study attributes of triangles. Using the new definition, this paper investigates the existence of the incenter and the isogonal conjugate of a triangle in Lorentzian plane.
\end{abstract}

\section{Introduction}
Plane geometry is one of the oldest branches of mathematics, and one of the basic studies in geometry involves triangles and their centers. The incenter is one of the classical triangle centers, discovered by the Greeks. The incenter is found by constructing the angle bisectors from each vertex and locating the intersection of the three bisectors. In fact, it can be proven that the three angle bisectors are concurrent, using Ceva's Theorem.

In Euclidean plane, angles are well-defined through magnitude of Euclidean rotation, and angle bisectors are given an intuitive, natural definition. However, in Lorentzian plane, it is difficult to define a similar measure of angle for any two intersecting lines. Nevertheless, it can be shown that incircle exists for certain type of triangles. In this paper, we will give a definition of angle bisectors so that incenters can be found using similar methods from Euclidean plane.

Furthermore, the angle bisectors are used to find other centers of any given triangle. One of these method involves finding the isogonal conjugate of any given triangle center. It follows that using the notion of angle bisectors in this paper, Lorentzian triangle centers also have an isogonal conjugate for certain types of triangles. Therefore, we will show that the new definition of angle bisectors can be used to show the existence of isogonal conjugates in Lorentzian plane.

\section{Euclidean angle bisector}

Bulk of the similarities between Euclidean and Lorentzian plane geometry comes from the fact that both geometries can be constructed through their respective inner products. Therefore, a common technique for translating properties of Euclidean geometry into those of Lorentzian geometry involves stating Euclidean facts in terms of Euclidean inner product, then replacing the Euclidean inner product by Lorentzian inner product. Therefore, we will develop a necessary and sufficient condition for angle bisectors in terms of Euclidean inner product.

\subsection{Properties of Euclidean inner product}
First, we state few well-known facts about Euclidean inner product. Given two vectors $\vec a = (a_1,a_2)$ and $\vec b = (b_1,b_2)$, the \emph{Euclidean inner product} is defined as
$$\vec a \bullet \vec b = a_1b_1 + a_2b_2.$$
From the definition, it follows that
\begin{equation} \label{eq:enorm}
  |\vec a| = \sqrt{\vec a \bullet \vec a}
\end{equation}
Furthermore, a common yet useful property says that
\begin{equation} \label{eq:edot}
  \vec a \bullet \vec b = |\vec a||\vec b|\cos\theta
\end{equation}
where $\theta$ is the angle between vectors $\vec a$ and $\vec b$.

\subsection{Redefining Euclidean angle bisector}
Using equation (1), we may express angle bisectors in terms of Euclidean inner product.

\begin{proposition}
  Let $\vec a$ and $\vec b$ be any two non-zero vectors. Then $\vec w$ is an angle bisector of $\vec a$ and $\vec b$ if and only if
  \begin{equation} \label{eq:eab}
    \frac{\vec a \bullet \vec w}{\sqrt{\vec a \bullet \vec a}} = \frac{\vec b \bullet \vec w}{\sqrt{\vec b \bullet \vec b}}.
  \end{equation}
\end{proposition}
\begin{proof}
  Let $\theta$ be the angle between $\vec a$ and $\vec w$ and $\varphi$ between $\vec b$ and $\vec w$. Note that \eqref{eq:enorm} and \eqref{eq:edot} implies that \eqref{eq:eab} is equivalent to the equation
  $$\cos\theta = \cos\varphi.$$
  However, since $0 \leq \theta, \varphi < \pi$, and cosine function is injective on the interval, the above equation is equivalent to
  $$\theta = \varphi.$$
  Therefore, (3) is equivalent to the fact that $\vec w$ is an angle bisector of $\vec a$ and $\vec b$.
\end{proof}
\noindent The expression \eqref{eq:eab} in Proposition 1 comes from the fact that orthogonal projection of unit vectors onto the angle bisector are equal, as shown in Figure \ref{fig:refl}. Note that in the figure, $\vec \alpha$ and $\vec \beta$ are unit vectors of $\vec a$ and $\vec b$, respectively.

\begin{figure}[ht]
    \centering
    \includegraphics[width=0.9\textwidth]{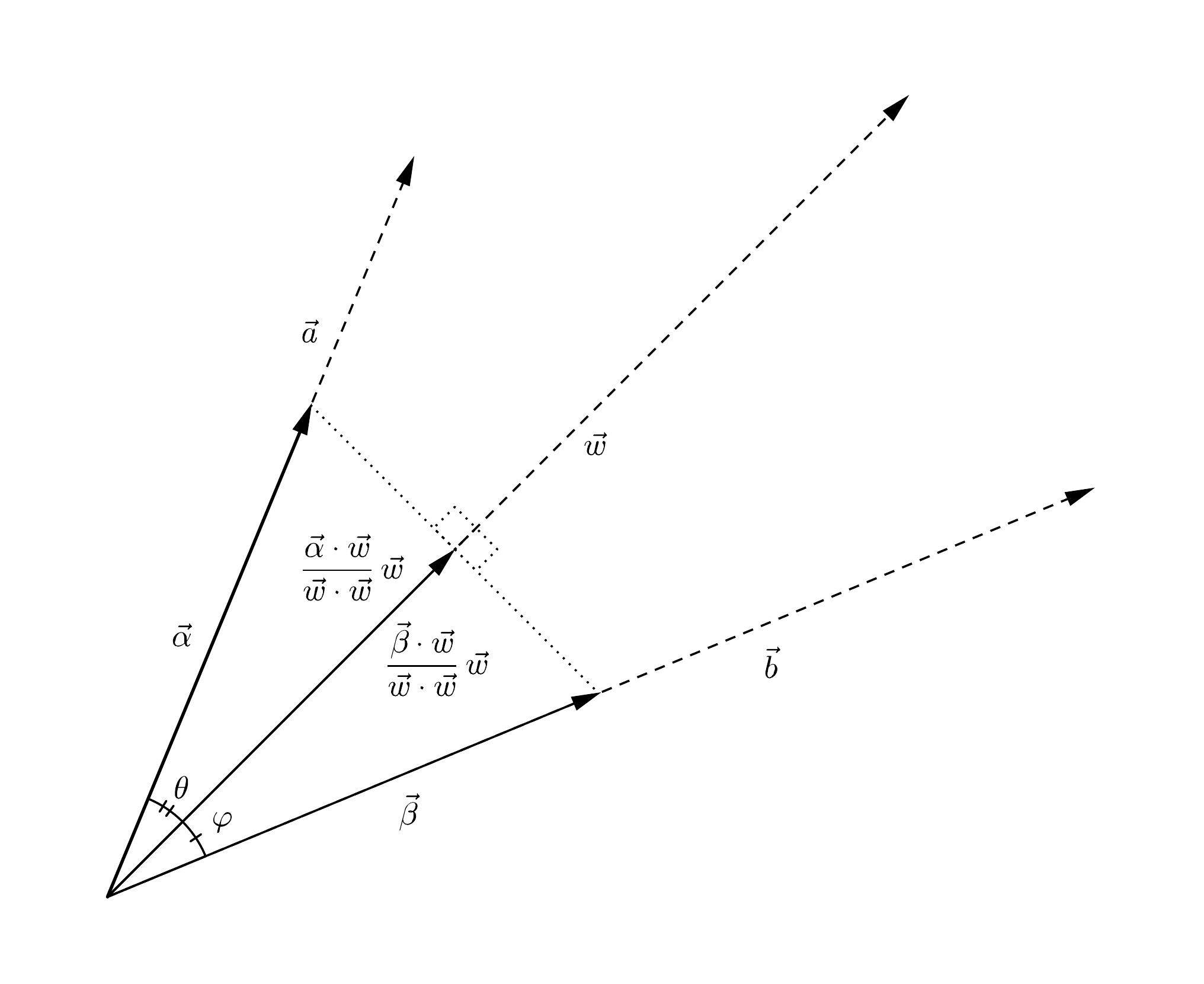}
    \caption{Geometrical representation of Proposition 1 \label{fig:refl}}
\end{figure}

\section{Lorentzian angle bisector}
As mentioned above, bulk of the technique of studying Lorentzian plane geometry comes from expressing concepts of Euclidean geometry in terms of Euclidean inner product, then simply switching the inner product to that of Lorentzian plane. Therefore, we must define the Lorentzian inner product. Furthermore, to address the need for a convoluted definition of angle bisectors, we define angles in Lorentzian plane.

\subsection{Lorentzian inner product and rotation}
Given two vectors $\vec a$ and $\vec b$ as above, the \emph{Lorentzian inner product} between $\vec a$ and $\vec b$ is defined as
$$\vec a \circ \vec b = a_1b_1 - a_2b_2,$$
and the \emph{Lorentzian magnitude} of a vector $\vec a$ given as
\begin{equation} \label{eq:lmag}
\|\vec a\| = \sqrt{|\vec a \circ \vec a|}.
\end{equation}
Note that (5) is different from the \emph{Lorentzian length} of a vector $\vec a$, defined as $$|\vec a| = \sqrt{\vec a \circ \vec a}.$$
Different from Euclidean distance, the squared Lorentzian distance may be positive, zero, or negative. We call non-zero vectors with positive, zero, or negative squared Lorentzian distance spacelike, lightlike, or timelike respectively. Such attribute of a Lorentzian vector is called the \emph{causal character}.

Since circles in Lorentzian plane is in the form of a Euclidean hyperbola, the rotation matrix takes the form of
\begin{equation} \label{eq:rotation}
  A[\varphi] = \begin{pmatrix}
\cosh\varphi & \sinh\varphi\\
\sinh\varphi & \cosh\varphi
\end{pmatrix}.
\end{equation}
By the evenness of hyperbolic cosine function and oddness of hyperbolic sine function,
\begin{equation} \label{eq:rotation2}
  (A[\varphi])^{-1} = A[-\varphi] = \begin{pmatrix}
\cosh\varphi & -\sinh\varphi\\
-\sinh\varphi & \cosh\varphi
\end{pmatrix}.
\end{equation}
One may immediately notice that given any two unit vectors, vectors with magnitude of one, the existence of a rotation taking one vector to another is not guaranteed. Therefore, the angle between two vectors is only defined when such a rotation exists between two vectors. Since the angle is not defined for any two vectors, the definition of angle bisectors applies to a very limited set. However, adopting the inner product approach of angle bisectors allows us to expand the scope of the definition.

\subsection{Lorentzian angle bisector}
Using the inner product definition of Euclidean angle bisector, we may now define Lorentzian angle bisector.
\begin{definition}
  Let $\vec a$ and $\vec b$ be two non-zero vectors. If there is some vector $\vec w$ such that
  \begin{equation} \label{eq:lab}
    \frac{\vec a \circ \vec w}{\sqrt{|\vec a \circ \vec a|}} = \frac{\vec b \circ \vec w}{\sqrt{|\vec b \circ \vec b|}}
  \end{equation}
  then we say $\vec w$ is the \emph{angle bisector} of $\vec a$ and $\vec b$.
\end{definition}
\noindent It should be noted that if $\vec\alpha$ and $\vec\beta$ are unit vectors of $\vec a$ and $\vec b$, respectively, then \eqref{eq:lab} may be rewritten as
$$\vec\alpha \circ \vec w = \vec\beta \circ \vec w.$$

It is easy to see that if two vectors $\vec w_1$ and $\vec w_2$ are both angle bisectors of a same pair of vector, then they are scalar multiples of each other. Also, note that the definition implies that the necessary condition for two vectors to have a angle bisector is to have the same causal character. Now, we will show that having the same causal character is also the sufficient condition.

\begin{proposition}
  Let $\vec a$ and $\vec b$ be two non-zero vectors sharing the same causal character. Let $\vec\alpha$ and $\vec\beta$ be the unit vectors of $\vec a$ and $\vec b$, respectively. Then $\vec\alpha + \vec\beta$ is an angle bisector of $\vec a$ and $\vec b$.
\end{proposition}
\begin{proof}
  Since $\vec a$ and $\vec b$ have the same causal character, it follows that
  $$\vec\alpha \circ \vec\alpha = \vec\beta \circ \vec\beta.$$
  By adding $\vec\alpha \circ \vec\beta$ on both sides, we get
  $$\vec\alpha \circ (\vec\alpha + \vec\beta) = \vec\beta \circ (\vec\alpha + \vec\beta)$$
  which can be rewritten as
  $$\frac{\vec a \circ (\vec\alpha + \vec\beta)}{\sqrt{|\vec a \circ \vec a|}} = \frac{\vec b \circ (\vec\alpha + \vec\beta)}{\sqrt{|\vec b \circ \vec b|}}.$$
\end{proof}

\begin{corollary}
  Let $\vec a$ and $\vec b$ be two non-zero vectors sharing the same causal character and the same orientation. If $\vec w$ is an angle bisector of $\vec a$ and $\vec b$, then $\vec w$ also have the same causal character and orientation.
\end{corollary}

It follows easily that every angle bisector is a scalar multiple of the vector sum of the unit vectors as above. Hence, the angle bisector must be either timelike or spacelike since the sum of two vectors of same causal character cannot yield a lightlike vector. This fact is important in the sense that we are guaranteed to have an isometry involving reflection with respect to an angle bisector.

The above definition comes directly from the inner product representation in the Euclidean case; however, it must agree with the natural definition of angle bisectors for vectors where angle may be defined. The next proposition will show that the above definition is equivalent to the intuitive definition for such vectors.

\begin{proposition}
  Let $\vec a$ and $\vec b$ be non-zero vectors with same causal character and same orientation, with unit vectors $\vec\alpha$ and $\vec\beta$, respectively. Then $\vec w$ is an angle bisector of $\vec a$ and $\vec b$ if and only if there is some rotation matrix $A[\varphi]$ such that $\vec w = A[\varphi]\vec\alpha$ and $\vec\beta = A[\varphi]\vec w$.
\end{proposition}
\begin{proof}
  Without loss of generality, we may assume that $\vec w$ is a unit vector since all angle bisectors of $\vec a$ and $\vec b$ are scalar multiples of each other. Furthermore, the proposition is trivial if $\vec\alpha = \vec\beta$. Therefore, we may assume that the two vectors are not scalar multiples of each other.

  First assume that $\vec w$ is an angle bisector of $\vec a$ and $\vec b$. Since the two vectors have the same causal character and orientation, $\vec w$ also have the same orientation, guaranteeing the existence of rotation matrices $A[\varphi_1]$ and $A[\varphi_2]$ such that
  \begin{equation} \label{eq:rot}
    \vec w = A[\varphi_1]\vec\alpha\text{ and }\vec\beta = A[\varphi_2]\vec w.
  \end{equation}
  Since $\vec\alpha \circ \vec w = \vec\beta \circ \vec w$, \eqref{eq:rot} implies that
  $$(A[-\varphi_1]\vec w - A[\varphi_2]\vec w)\circ\vec w = 0.$$
  Therefore, for some real $s$,
  $$A[-\varphi_1]\vec w - A[\varphi_2]\vec w = s(w_2,w_1)$$
  where $\vec w = (w_1,w_2)$.
  By applying \eqref{eq:rotation} and \eqref{eq:rotation2} and simplifying, we find that
  \begin{align*}
    (\cosh\varphi_1-\cosh\varphi_2)w_1 - (\sinh\varphi_1 + \sinh\varphi_2)w_2 &= sw_2\\
    -(\sinh\varphi_1 + \sinh\varphi_2)w_1 + (\cosh\varphi_1- \cosh\varphi_2)w_2 &= sw_1,
  \end{align*}
  implying that $\cosh\varphi_1 = \cosh\varphi_2$. Since $\vec\alpha$ and $\vec\beta$ are distinct, \eqref{eq:rot} implies that $\varphi_1 = \varphi_2$.

  On the other hand, assume that there exists a rotation matrix $A[\varphi]$ such that
  $$\vec w = A[\varphi]\vec\alpha\text{ and }\vec\beta = A[\varphi]\vec w.$$
  Then,
  \begin{equation} \label{eq:checkref}
    \vec\alpha \circ \vec w - \vec\beta \circ \vec w = (A[-\varphi]\vec w - A[\varphi]\vec w)\circ \vec w.
  \end{equation}
  By applying (6) and (7) and simplifying,
  $$A[-\varphi]\vec w - A[\varphi]\vec w = -2\sinh\varphi(w_2,w_1)$$
  where $\vec w = (w_1,w_2)$. Hence, the expression \eqref{eq:checkref} equals zero.
\end{proof}

\section{Preliminaries from Lorentzian plane geometry}
Before we proceed to prove more interesting results, we must first state few results from Lorentzian plane geometry. Note that from this point on, given a triangle $ABC$, $\vec a$ signifies the vector opposite of vertex $A$, and the direction is given so that $\vec a + \vec b + \vec c = 0$. In Euclidean geometry, given two vectors $\vec a$ and $\vec b$ and an angle $\theta$ between the two vectors, (2) can be used so that
$$\sin^2\theta = 1-\frac{\vec a \bullet \vec b}{(\vec a \bullet \vec a)(\vec b \bullet \vec b)}.$$
Therefore, we may define a Lorentzian version of squared trigonometric function, as follows:
$$S^2_{ab} = 1- \frac{\vec a \circ \vec b}{(\vec a \circ \vec a)(\vec b \circ \vec b)}.$$
Now we may state the \emph{law of sines} in Lorentzian case using the above notation.
\begin{theorem}[Lorentzian law of sines]
  Let $ABC$ be a triangle without null edges. Then,
  \begin{equation} \label{eq:sines}
    \frac{S^2_{bc}}{|\vec a|^2} = \frac{S^2_{ca}}{|\vec b|^2} = \frac{S^2_{ab}}{|\vec c|^2}.
  \end{equation}
\end{theorem}

Also, we need to state the Lorentzian version of Ceva's theorem and its converse, to be used in the proofs of existence of incenter and the isogonal conjugate.
\begin{theorem}[Lorentzian Ceva's Theorem]
  Let $ABC$ be a triangle without null edges. Let $P$ be any point, and let $D$, $E$, and $F$, be the intersection between line $AP$ and $BC$, $BP$ and $CA$, and $CP$ and $AB$, respectively. Then,
  \begin{equation} \label{eq:ceva1}
    \frac{|\overrightarrow{AF}|}{|\overrightarrow{FB}|}
    \cdot\frac{|\overrightarrow{BD}|}{|\overrightarrow{DC}|}
    \cdot\frac{|\overrightarrow{CE}|}{|\overrightarrow{EA}|} = 1.
  \end{equation}
\end{theorem}
\begin{theorem}[Converse to Ceva's Theorem]
  Let $ABC$ be a triangle without null edges. If $D$, $E$, and $F$ are points on line $BC$, $CA$, and $AB$, respectively, such that \eqref{eq:ceva1} holds, then lines $AD$, $BE$, and $CF$ are concurrent.
\end{theorem}
\noindent The proofs to the above theorems is exactly the same as the Euclidean counterpart using the area of a triangle, since the area of a Lorentzian triangle without null edges is also equal to one-half of base times height. Note that the segments in the above fraction have the same causal character pairwise. Therefore, \eqref{eq:ceva1} is equivalent to
\begin{equation} \label{eq:ceva2}
  \frac{\overrightarrow{AF}\circ\overrightarrow{AF}}{\overrightarrow{FB}\circ\overrightarrow{FB}}
  \cdot\frac{\overrightarrow{BD}\circ\overrightarrow{BD}}{\overrightarrow{DC}\circ\overrightarrow{DC}}
  \cdot\frac{\overrightarrow{CE}\circ\overrightarrow{CE}}{\overrightarrow{EA}\circ\overrightarrow{EA}}
  = \frac{|\overrightarrow{AF}|^2}{|\overrightarrow{FB}|^2}
  \cdot\frac{|\overrightarrow{BD}|^2}{|\overrightarrow{DC}|^2}
  \cdot\frac{|\overrightarrow{CE}|^2}{|\overrightarrow{EA}|^2} = 1.
\end{equation}
Since we primarily express attributes of Lorentzian triangle in terms of inner product, we will make use of \eqref{eq:ceva2} more frequently than \eqref{eq:ceva1}. The notation for Ceva's theorem are illustrated in Figure \ref{fig:ceva}.
\begin{figure}[ht]
    \centering
    \includegraphics[width=0.9\textwidth]{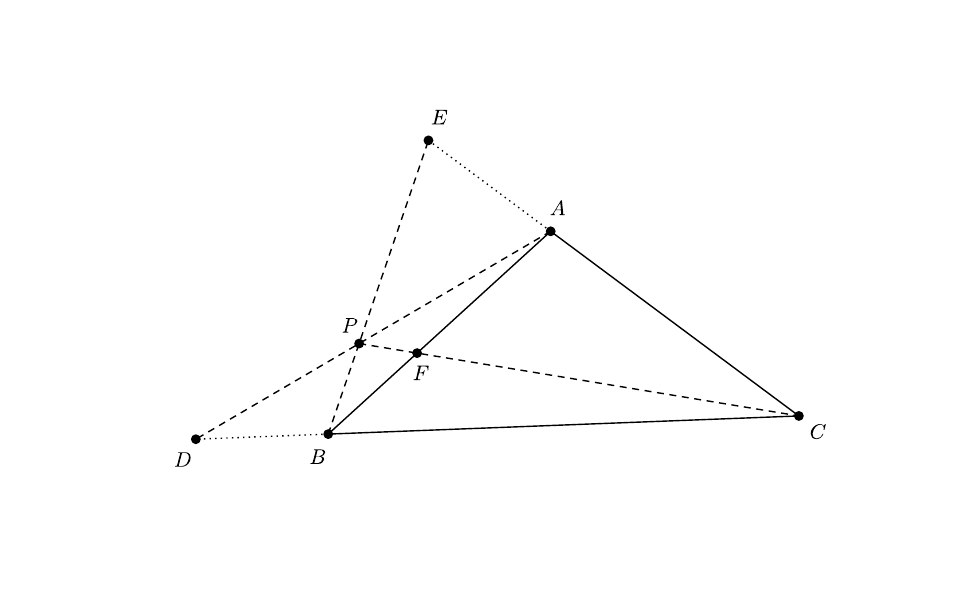}
    \caption{Illustration of Ceva's theorem where $P$ is outside the triangle\label{fig:ceva}.}
\end{figure}

\section{Incenter of a Lorentzian triangle}
In a Euclidean triangle, the incenter is located by finding the intersection of the three angle bisectors. The existence of angle bisectors was guaranteed for any triangle in Euclidean plane; therefore, the incenter always exist. However, in Lorentzian plane, the existence of angle bisector is conditional, namely, the two vectors associated must have the same causal character. Therefore, angle bisectors exist at every vertex of a triangle if and only if the triangle is a \emph{pure triangle}, a triangle whose sides share a single causal character.

In fact, one should notice that Lorentzian circles are either horizontal or vertical Euclidean hyperbolas. Therefore, all tangent lines to a given Lorentzian circle are either timelike or spacelike. It follows that for a triangle to have a Euclidean hyperbola that is tangent to every side, then the triangle must be a pure triangle. Thus, the above restriction from our definition of angle bisectors is justified, and we will assume from this point on that triangle $ABC$ is a pure triangle unless otherwise stated.

To prove that the angle bisectors are concurrent, we follow the proof in Euclidean case. The Euclidean case uses bisector theorem and applies Ceva's theorem to show that the concurrency of the angle bisectors. Hence, we must state and prove the bisector theorem for Lorentzian case. The notations of the theorem are illustrated in Figure \ref{fig:bisector}.

\begin{figure}[ht]
    \centering
    \includegraphics[width=0.8\textwidth]{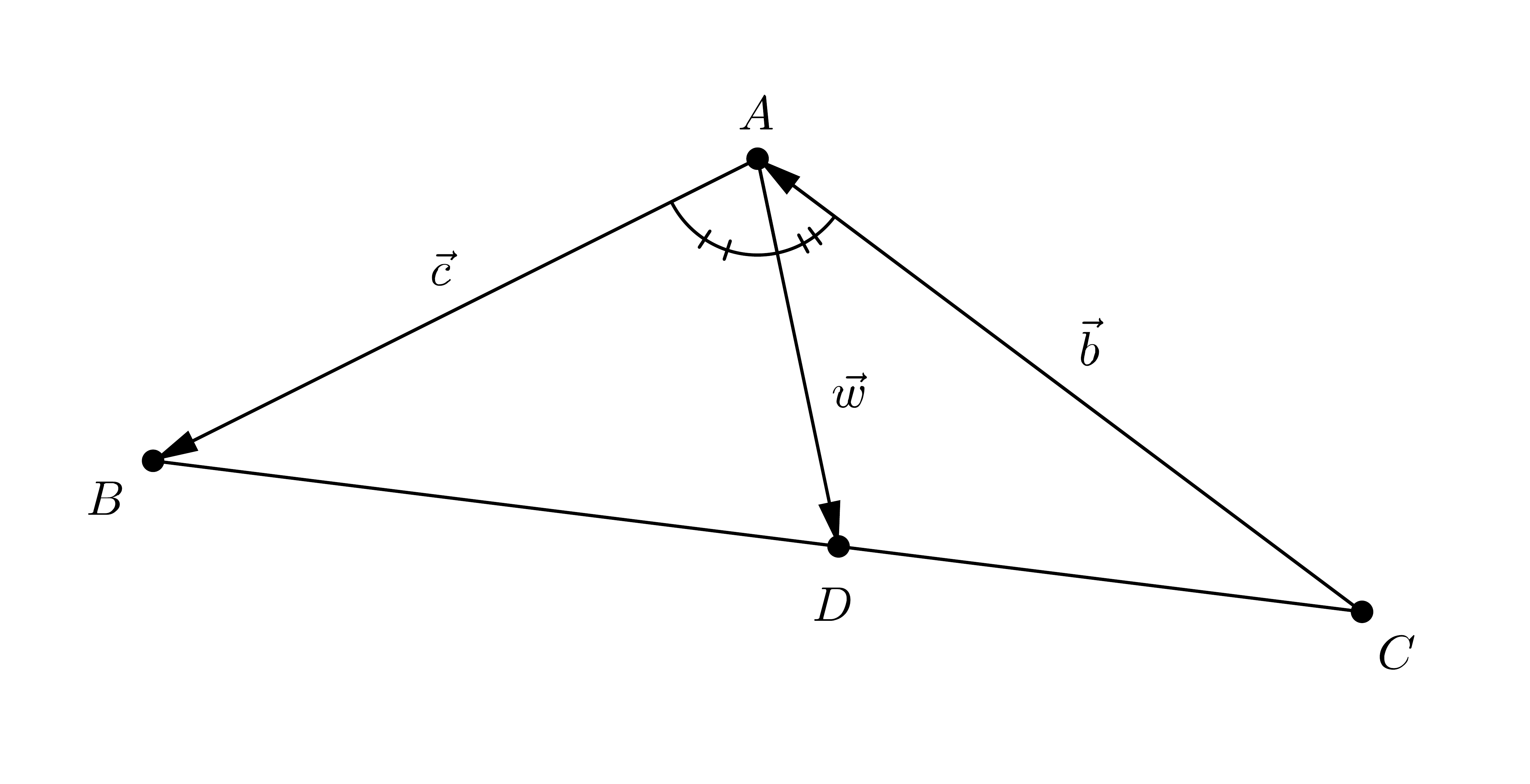}
    \caption{Lorentzian angle bisector of a triangle\label{fig:bisector}.}
\end{figure}

\begin{theorem}[Lorentzian bisector theorem]
  Given a triangle $ABC$, let $D$ be on $BC$ such that $\overrightarrow{AD} = \vec w$ is an angle bisector of $\vec b$ and $\vec c$. Then,
  \begin{equation} \label{eq:bisector}
    \frac{|\vec b|^2}{|\vec c|^2} = \frac{|\overrightarrow{DC}|^2}{|\overrightarrow{BD}|^2}.
  \end{equation}
\end{theorem}
\begin{proof}
  Since $\vec w$ is an angle bisector of $\vec b$ and $\vec c$,
  $$\frac{\vec b \circ \vec w}{\sqrt{|\vec b\circ\vec b|}} = \frac{\vec c \circ \vec w}{\sqrt{|\vec c\circ\vec c|}}$$
  and squaring both sides tells us that
  $$\frac{(\vec b \circ \vec w)^2}{(\vec b\circ\vec b)(\vec w \circ \vec w)} = \frac{(\vec c \circ \vec w)^2}{(\vec c\circ\vec c)(\vec w\circ\vec w)}$$
  where the absolute value is taken off since $\vec b$ and $\vec c$ have the same causal character.
  Hence,
  \begin{equation}\label{eq:angle1}
    S^2_{bw} = S^2_{cw}.
  \end{equation}
  On the other hand, since $\overrightarrow{BD}$ and $\overrightarrow{DC}$ are scalar multiples of $\vec a$, it can be easily checked that
  \begin{equation}\label{eq:angle2}
    S^2_{DCb} = S^2_{ab}\text{ and }S^2_{BDc} = S^2_{ac}.
  \end{equation}
  Now, applications of law of sines to triangles $ABD$ and $ACD$ implies
  $$\frac{|\overrightarrow{DC}|^2}{|\vec w|^2} = \frac{S^2_{bw}}{S^2_{DCb}}\text{ and }
  \frac{|\vec w|^2}{|\overrightarrow{BD}|^2} = \frac{S^2_{BDc}}{S^2_{cw}}.$$
  Therefore, by applying \eqref{eq:angle1} and \eqref{eq:angle2},
  $$\frac{|\overrightarrow{DC}|^2}{|\overrightarrow{BD}|^2} = \frac{S^2_{bw}}{S^2_{DCb}}\cdot\frac{S^2_{BDc}}{S^2_{cw}} = \frac{S^2_{ac}}{S^2_{ab}} = \frac{|\vec b|^2}{|\vec c|^2}$$
  where the last inequality comes from applying law of sines to the triangle $ABC$.
\end{proof}

Now the concurrency of the angle bisectors follows from a direct application of bisector theorem and Ceva's theorem.

\begin{theorem}[Incenter of a pure triangle]
  Given a triangle $ABC$, let $D$, $E$, and $F$ be on $BC$, $CA$, and $AC$ such that $\overrightarrow{AD}$, $\overrightarrow{BE}$, and $\overrightarrow{CF}$ are angle bisectors, respectively. Then $AD$, $BE$, and $CF$ are concurrent.
\end{theorem}
\begin{proof}
  Note that since $\overrightarrow{AD}$, $\overrightarrow{BE}$, and $\overrightarrow{CF}$ are angle bisectors, \eqref{eq:bisector} implies that, \eqref{eq:ceva2} is equal to
  $$\frac{|\vec b|^2}{|\vec a|^2}\cdot
\frac{|\vec c|^2}{|\vec b|^2}\cdot
\frac{|\vec b|^2}{|\vec c|^2} = 1.$$
  Therefore, by the converse of Ceva's theorem, $AD$, $BE$ and $CF$ are concurrent.
\end{proof}

It can be checked via simple calculation that the distances from the point of concurrency to the sides are all equal. However, the calculation is tedius and long, which can be done by a computer software. We leave the search a more elegant proof for future studies and accept that the point of concurrency found above is the incenter of the given triangle. Figure \ref{fig:lincenter} depicts the angle bisectors and the incircle of triangle $ABC$.

\begin{figure}[ht]
    \centering
    \includegraphics[width=0.8\textwidth]{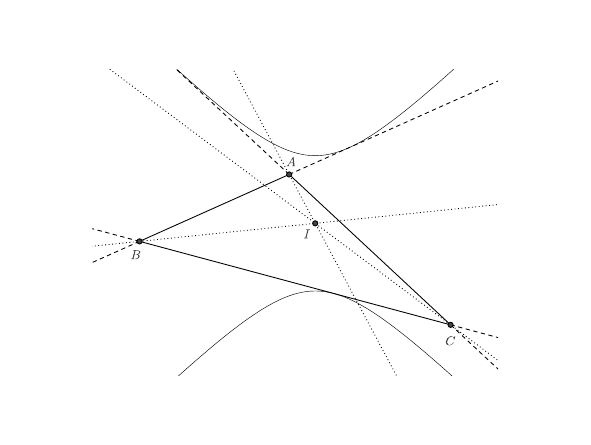}
    \caption{Angle bisectors and incenter of a pure triangle\label{fig:lincenter}.}
\end{figure}

\section{Isogonal conjugate of a Lorentzian triangle center}
A perhaps more complex application of angle bisectors is the concept of isogonal conjugacy. First we depict isogonal conjugacy in Euclidean plane through Figure \ref{fig:eisogonal}. Given a point $P$, we construct three cevians such that each line passes through a vertex and the given point. If the cevians are reflected across their respective angle bisectors, then the new cevians be concurrent at a point, $P'$. Such $P'$ is called the \emph{isogonal conjugate} of $P$. Note that in Figure \ref{fig:eisogonal}, to each angle bisectors, there are two cevians that are reflections across the angle bisector.

\begin{figure}[ht]
    \centering
    \includegraphics[width=0.8\textwidth]{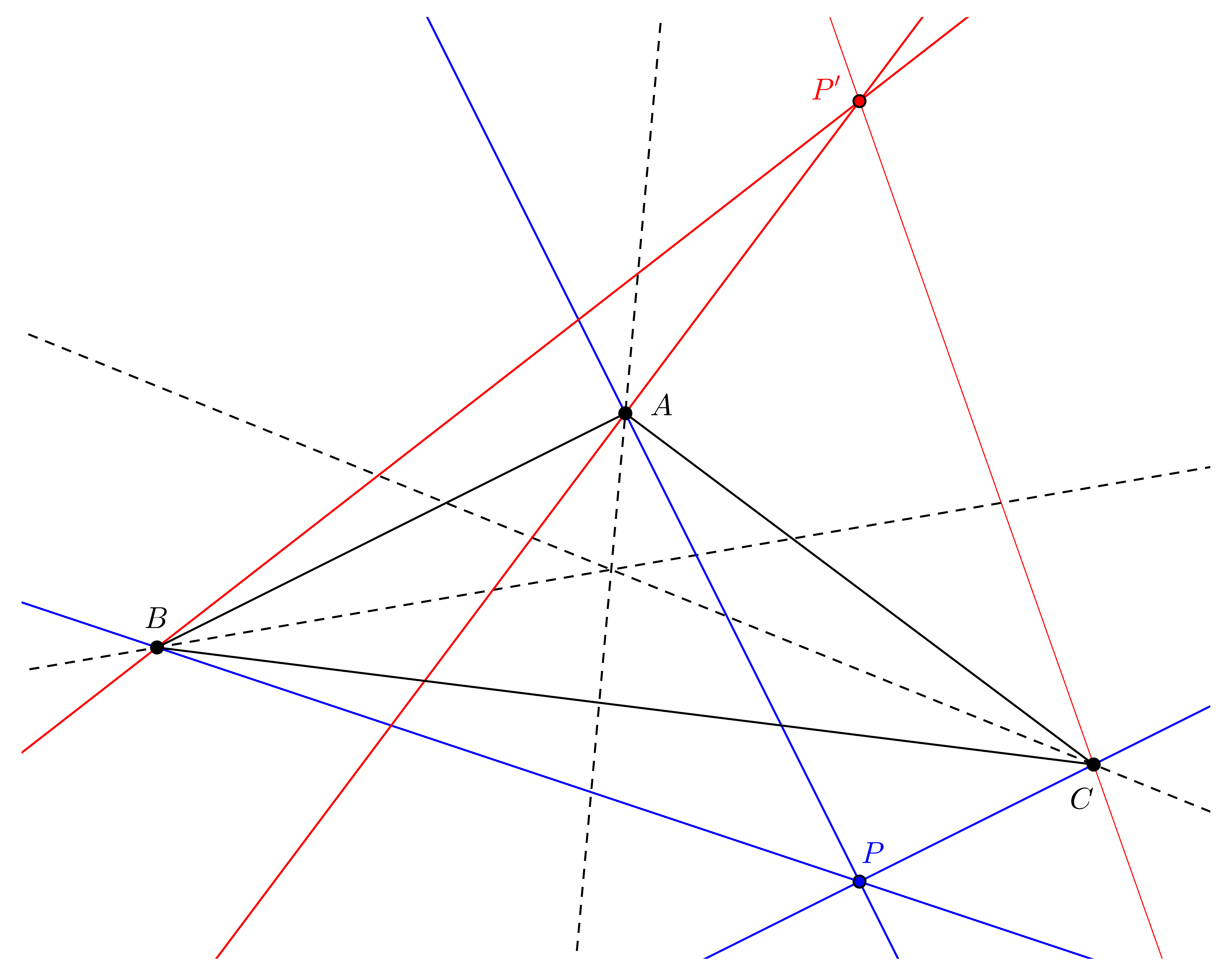}
    \caption{Isogonal conjugate $P'$ of $P$ in Euclidean plane\label{fig:eisogonal}.}
\end{figure}

We want to prove that isogonal conjugates also exist in Lorentzian plane geometry. However, the process of locating isogonal conjugate involves reflection across a line. In Lorentzian plane, reflections across a given line may not exist; to be precise, we cannot find an isometry that involves a reflection across a lightlike line. Nevertheless, as mentioned earlier, angle bisectors are either timelike or spacelike, allowing us to discuss reflections across angle bisectors.

Note that following the terminology from Euclidean case, if two cevians are reflections of each other across an angle bisector, then we say that the cevians are \emph{isogonal}. Again, since we need to make use of angle bisectors, we must impose the restriction that the triangles in consideration must be pure triangles. In proving the concurrency of the reflected cevians, the following lemma is crucial. Figure \ref{fig:lemma} explains the notations concerning lemma.
\newpage
\begin{lemma}
  Given a triangle $ABC$, let $D$ and $E$ be on $BC$ such that $AD$ and $AE$ are isogonals. Furthermore, let $\overrightarrow{AF} = \vec w$ be the angle bisector of $\vec b$ and $\vec c$. Then,
  $$\frac{|\overrightarrow{BD}|^2}{|\overrightarrow{DC}|^2}\cdot
  \frac{|\overrightarrow{BE}|^2}{|\overrightarrow{EC}|^2}
  = \frac{|\vec c|^4}{|\vec b|^4}.$$
\end{lemma}
\begin{proof}
  Since $AE$ is the angle bisector of $\vec b$ and $\vec c$,
  $$\frac{\vec c\circ\vec w}{\sqrt{|\vec c \circ\vec c|}}
     = \frac{\vec b\circ\vec w}{\sqrt{|\vec b \circ\vec b|}}$$
  which can be rewritten as
  \begin{equation}\label{eq:lma1}
    \left(\frac{\vec c}{\sqrt{|\vec c \circ\vec c|}}
     - \frac{\vec b}{\sqrt{|\vec b \circ\vec b|}}\right)\circ\vec w = 0.
  \end{equation}
  Similarly, since $AD$ and $AE$ are isogonal,
  \begin{equation}\label{eq:lma2}
    \left(\frac{\overrightarrow{AD}}{\sqrt{|\overrightarrow{AD}\circ\overrightarrow{AD}|}}
     - \frac{\overrightarrow{AE}}{\sqrt{|\overrightarrow{AE}\circ\overrightarrow{AE}|}}\right)\circ\vec w = 0.
  \end{equation}
  Therefore, \eqref{eq:lma1} and \eqref{eq:lma2} implies that
  \begin{equation}\label{eq:lma3}
    \left(\frac{\overrightarrow{AD}}{\sqrt{|\overrightarrow{AD}\circ\overrightarrow{AD}|}}
     - \frac{\overrightarrow{AE}}{\sqrt{|\overrightarrow{AE}\circ\overrightarrow{AE}|}}\right) =
     s\left(\frac{\vec c}{\sqrt{|\vec c \circ\vec c|}}
     - \frac{\vec b}{\sqrt{|\vec b \circ\vec b|}}\right)
  \end{equation}
  for some non-zero real $s$. By rewriting \eqref{eq:lma3} and taking the inner product of itself, we find out that
  \begin{align*}
    \frac{\overrightarrow{AD}\circ\overrightarrow{AD}}{|\overrightarrow{AD}\circ\overrightarrow{AD}|}
  &-2s\frac{\overrightarrow{AD}\circ\vec c}{\sqrt{|\overrightarrow{AD}\circ\overrightarrow{AD}||\vec c \circ\vec c|}}
  +s^2\frac{\vec c \circ\vec c}{|\vec c \circ\vec c|}\\
  &=\frac{\overrightarrow{AE}\circ\overrightarrow{AE}}{|\overrightarrow{AE}\circ\overrightarrow{AE}|}
  -2s\frac{\overrightarrow{AE}\circ\vec b}{\sqrt{|\overrightarrow{AE}\circ\overrightarrow{AE}||\vec b \circ\vec b|}}
  +s^2\frac{\vec b \circ\vec b}{|\vec b \circ\vec b|}.
  \end{align*}
  Since $\overrightarrow{AD}$ and $\overrightarrow{AE}$, $\vec b$ and $\vec c$ share the same causal character pairwise, the above equation simplies to
  \begin{equation}\label{eq:lma4}
    S^2_{ADc} = S^2_{AEb}.
  \end{equation}
  Using similar technique after rewriting \eqref{eq:lma3} in a different way from above tells us that
  \begin{equation}\label{eq:lma5}
    S^2_{ADb} = S^2_{AEc}.
  \end{equation}
  On the other hand, since $\overrightarrow{BD}$, $\overrightarrow{BE}$, $\overrightarrow{DC}$, and $\overrightarrow{EC}$ are all scalar multiples of $\vec a$,
  \begin{equation}\label{eq:lma8}
    S^2_{BDc} = S^2_{BEc} = S^2_{ac}\text{ and }
    S^2_{DCb} = S^2_{ECb} = S^2_{ab}
  \end{equation}
  Now apply law of sines to triangles $ABD$ and $ADC$ to get
  $$\frac{|\overrightarrow{BD}|^2}{|\overrightarrow{AD}|^2}
  = \frac{S^2_{ADc}}{S^2_{BDc}}
  \text{ and }
  \frac{|\overrightarrow{AD}|^2}{|\overrightarrow{DC}|^2}
  = \frac{S^2_{CDb}}{S^2_{ADb}}$$
  which in turn implies that
  \begin{equation}\label{eq:lma6}
    \frac{|\overrightarrow{BD}|^2}{|\overrightarrow{DC}|^2}
    =\frac{S^2_{ADc}}{S^2_{BDc}}\cdot\frac{S^2_{CDb}}{S^2_{ADb}}.
  \end{equation}
  Similarly, applying law of sines to triangles $ABE$ and $AEC$ shows us that
  \begin{equation}\label{eq:lma7}
    \frac{|\overrightarrow{BE}|^2}{|\overrightarrow{EC}|^2}
    =\frac{S^2_{AEc}}{S^2_{BEc}}\cdot\frac{S^2_{CEb}}{S^2_{AEb}}.
  \end{equation}
  Finally, \eqref{eq:lma4}, \eqref{eq:lma5}, and \eqref{eq:lma8} implies that
  $$\frac{|\overrightarrow{BD}|^2}{|\overrightarrow{DC}|^2}
  \cdot \frac{|\overrightarrow{BE}|^2}{|\overrightarrow{EC}|^2}
  =\frac{(S^2_{ab})^2}{(S^2_{ac})^2} = \frac{|\vec c|^4}{|\vec b|^4}$$
  where the last inequality comes from applying law of sines to the triangle $ABC$.
\end{proof}
\begin{figure}[ht]
    \centering
    \includegraphics[width=0.8\textwidth]{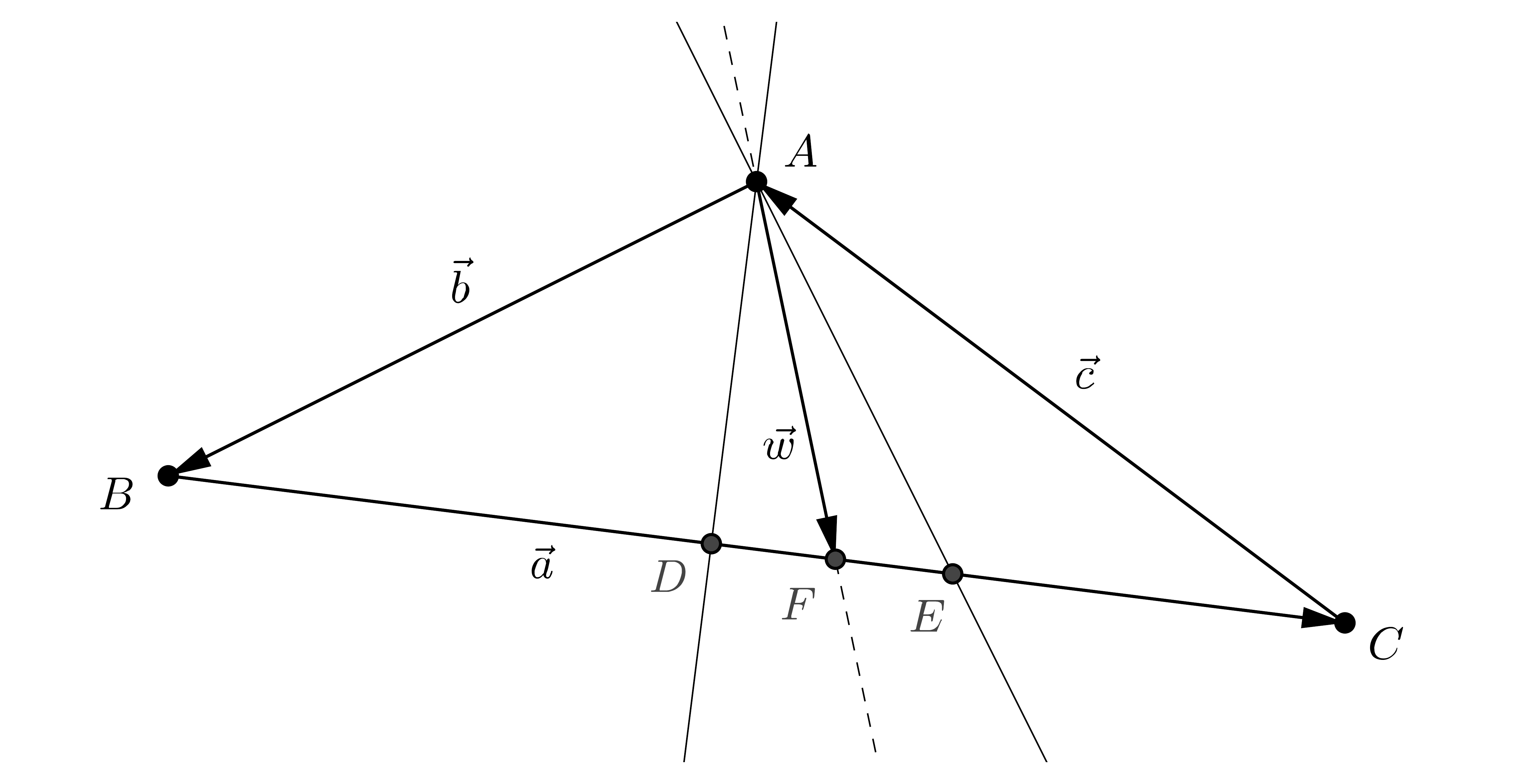}
    \caption{Pair of isogonal cevians $AD$ and $AE$ and the angle bisector $AF$\label{fig:lemma}.}
\end{figure}

Now we may proceed to prove that isogonal conjugates exist in Lorentzian plane for pure triangles. The notations for the following theorem are represented in Figure \ref{fig:lisogonal}.

\begin{figure}[ht]
    \centering
    \includegraphics[width=0.8\textwidth]{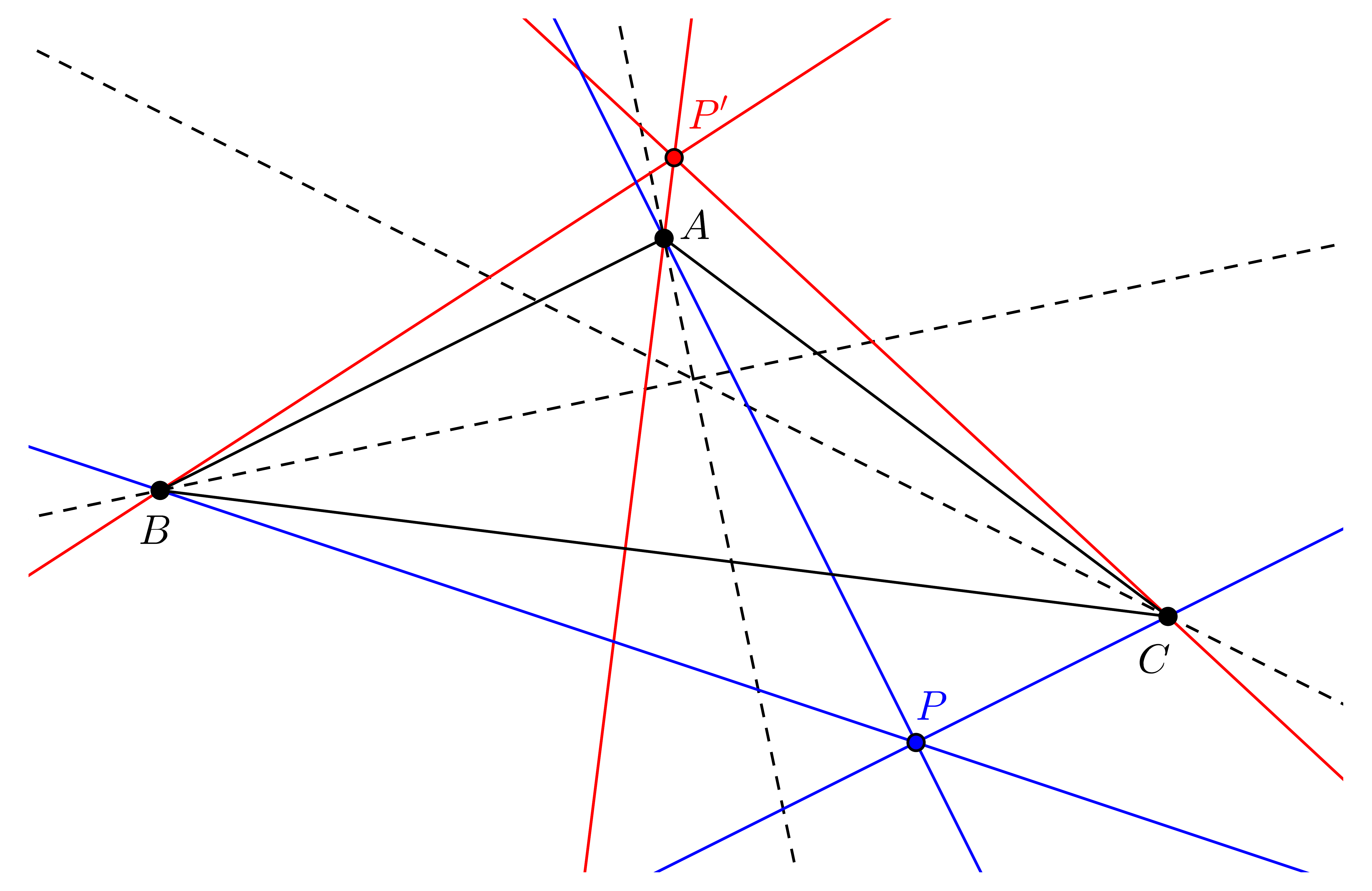}
    \caption{Isogonal conjugate $P'$ of $P$ in Lorentzian plane\label{fig:lisogonal}.}
\end{figure}
\newpage
\begin{theorem}[Isogonal conjugate theorem]
  Let triangle $ABC$ be given. Pick any point $P$ such that the cevians can be drawn from each vertex onto the opposite side. Let $D$, $E$, and $F$ be the intersections between $BC$ and $AP$, $CA$ and $BP$, $AB$ and $CP$, respectively. If $D'$, $E'$, $F'$ are points on $BC$, $CA$, and $AB$ such that $AD$ and $AD'$, $AE$ and $AE'$, $AF$ and $AF'$ are isogonal respectively, then $AD'$, $AE'$ and $AF'$ are concurrent.
\end{theorem}
\begin{proof}
  By the above lemma,
  \begin{align*}
    \frac{|\overrightarrow{BD}|^2}{|\overrightarrow{DC}|^2}
    \cdot\frac{|\overrightarrow{BD'}|^2}{|\overrightarrow{D'C}|^2}
    = \frac{|\vec c|^4}{|\vec b|^4},
    \frac{|\overrightarrow{CE}|^2}{|\overrightarrow{EA}|^2}
    \cdot\frac{|\overrightarrow{CE'}|^2}{|\overrightarrow{E'A}|^2}
    = \frac{|\vec a|^4}{|\vec c|^4},\text{ and }
    \frac{|\overrightarrow{AF}|^2}{|\overrightarrow{FB}|^2}
    \cdot\frac{|\overrightarrow{AF'}|^2}{|\overrightarrow{F'B}|^2}
    = \frac{|\vec b|^4}{|\vec a|^4}
  \end{align*}
  implying that
  $$\frac{|\overrightarrow{BD}|^2}{|\overrightarrow{DC}|^2}
    \cdot\frac{|\overrightarrow{BD'}|^2}{|\overrightarrow{D'C}|^2}
    \cdot\frac{|\overrightarrow{CE}|^2}{|\overrightarrow{EA}|^2}
    \cdot\frac{|\overrightarrow{CE'}|^2}{|\overrightarrow{E'A}|^2}
    \cdot\frac{|\overrightarrow{AF}|^2}{|\overrightarrow{FB}|^2}
    \cdot\frac{|\overrightarrow{AF'}|^2}{|\overrightarrow{F'B}|^2} = 1.$$
  However, since $AD$, $BE$, and $CF$ are concurrent at point $P$, by Ceva's theorem,
  $$\frac{|\overrightarrow{BD}|^2}{|\overrightarrow{DC}|^2}
    \cdot\frac{|\overrightarrow{CE}|^2}{|\overrightarrow{EA}|^2}
    \cdot\frac{|\overrightarrow{AF}|^2}{|\overrightarrow{FB}|^2} = 1.$$
  Therefore,
  $$\frac{|\overrightarrow{BD'}|^2}{|\overrightarrow{D'C}|^2}
    \cdot\frac{|\overrightarrow{CE'}|^2}{|\overrightarrow{E'A}|^2}
    \cdot\frac{|\overrightarrow{AF'}|^2}{|\overrightarrow{F'B}|^2} = 1$$
  and applying the converse of Ceva's theorem gives us the desired result.
\end{proof}

The point of concurrency of the new cevians is called the \emph{isogonal conjugate} of the point $P$. Note that the triangle $ABC$ and point $P$ are identical in Figure \ref{fig:eisogonal} and Figure \ref{fig:lisogonal}. Expectedly, it is evident that the angle bisectors and the reflected cevians are different, yet in each case the new cevians are concurrent. Now that we have proven the existence of isogonal conjugates, the existence of \emph{Lemoine point} or \emph{symmedian point} for pure triangles in Lorentzian plane follows immediately, since by definition, Lemoine point is the isogonal conjugate of centroid.

\section{Conclusion}
In Lorentzian plane, discussion of angles is complicated by the fact that a rotation matrix might not exist for any given pair of unit vectors. In fact, if angles are defined using the Lorentzian rotation matrix, then every triangle will have at least one undefined angle. Therefore, it is difficult to utilze angle bisectors in discussing attributes of triangles in Lorentzian plane.

However, by adopting a slightly modified definition of angle bisectors introduced in this article, one that extends the natural definition of angle bisectors where angles are defined, then we are able to analyze pure triangles using angle bisectors. Therefore, the discussion on existence of incenters and isogonal conjugates naturally follows, and similar to the Euclidean counterpart, existence of both elements is guaranteed for pure triangles. Also, the existence of Lemoine point follows directly as a result.

The theory of isogonal conjugates have other interesting properties in that several well-known triangle centers are isogonal conjugates of each other. In Euclidean plane, the circumcenter and the orthocenter are a known example. Also the mittenpukt is defined as the symmedian point of the excentral triangle. Future studies can further develop the theory of isogonal conjugates by translating the facts such as above from Euclidean plane to the Lorentzian plane.

\end{document}